\documentclass[12pt]{amsart}
\usepackage[english]{babel}
\usepackage{amsmath,amssymb,amsfonts,amsthm,latexsym} 
\usepackage{float}	
\usepackage{enumitem}
\usepackage{tikz}
\usetikzlibrary{calc}

\newtheorem*{maintheorem}{Theorem}

\newtheorem{proposition}{Proposition}

\theoremstyle{remark}

\theoremstyle{definition}

\newcommand{\bR}{\mathbb{R}}
\newcommand{\bZ}{\mathbb{Z}}
\newcommand{\bN}{\mathbb{N}}

\newcommand{\g}{\mathbf{\Gamma}}
\newcommand{\uL}{\mathbf{L}}
\newcommand{\ux}{\mathbf{x}}

\begin{document}

\baselineskip=17pt

\title[Parametric geometry of numbers ]
{Continued fractions\\ and \\ Parametric geometry of numbers}

\author{Aminata Dite Tanti Keita}

\address{
   D\'epartement de Math\'ematiques\\
   Universit\'e d'Ottawa\\
   585 King Edward\\
   Ottawa, Ontario K1N 6N5, Canada}
  \email{akeit104@uottawa.ca}
\subjclass[2010]{Primary 11J04; Secondary 11J82}
\thanks{Work partially supported by NSERC (Canada)}

\begin{abstract}
Recently, W.~M.~Schmidt and L.~Summerer developed a new theory called 
Parametric Geometry of Numbers which approximates the behaviour 
of the successive minima of a family of convex bodies in  $\mathbb{R}^{n}$ 
related to the problem of simultaneous rational 
approximation to given real numbers. In the case of one number, we show 
that the qualitative behaviour of the  minima reflects the
continued fraction expansion of the smallest distance from this number
 to an integer.
\end{abstract}

\maketitle

\section{Introduction}
The parametric geometry of numbers, recently introduced by W.\ M.\ Schmidt and L.\ Summerer,
is a theory which analyzes the behaviour of the successive minima of parametric families of convex bodies in
 $\bR^n$. It provides a new approach to
 the problem of simultaneous approximation to real numbers by rational numbers. It was initially 
developed in dimension $n=3$ in \cite{SS2009}, then
extended to the general case $n\geqslant 2$ in \cite{SS2013a}, and completed in
 \cite{R_preprint}. Our goal is to revisit the case $n=2$ by providing 
an exact description of the qualitative behaviour of the successive minima in that case. 
For this, we consider the family of convex bodies given by
\[
 \mathcal{C}_{\xi}(e^{q}) := \lbrace (x,y) \in \mathbb{R}^{2}\,;\,|x| \leqslant e^{q} \,,\,|x\xi -y| 
 \leqslant e^{-q} \rbrace \qquad ( q\geqslant 0 ), 
\]
for some fixed $\xi\in\bR$. For $j=1,2$ and $q\geqslant0$, let $L_{\xi,j}(q)=
 \log \lambda_{j}(\mathcal{C}_{\xi}(e^{q}))$, where $\lambda_{j}(\mathcal{C}_{\xi}(e^{q}))$ 
denotes the $j$-th minimum of $ \mathcal{C}_{\xi}(e^{q})$, i.e.\ the smallest 
$\lambda \geqslant 0$ such that $\lambda 
\mathcal{C}_{\xi}(e^{q})$ contains at least $j$ linearly independent elements 
in $\mathbb{Z}^{2}$. Further, we define a function $\uL_{\xi}: [0,\infty)\longrightarrow \mathbb{R}$ by
\[\uL_{\xi}(q)= (L_{\xi,1}(q),L_{\xi,1}(q)) \qquad (q\geqslant 0).\]
As noted by Schmidt and Summerer \cite[$\mathsection 4$]{SS2009}, $\uL_{\xi}$
enjoys the following properties.
\begin{itemize}
 \item[\huge .]
 Each component $L_{\xi,j}: [0,\infty)\rightarrow \mathbb{R}$ is continuous 
and piecewise linear with slope $\pm1$.
\item[\huge .]
If $q\in [0,\infty)$ is such that $L_{\xi,1}$ admits a local maximum at $q$, then we have \[L_{\xi,1}(q)=~L_{\xi,2}(q).\]
\item[\huge .]
For all $q\geqslant0$, we have -$\log2\leqslant L_{\xi,1}(q)+L_{\xi,2}(q) \leqslant 0$ 
(by virtue of Minkowski's theorem.)
\end{itemize}
The union of the graphs of $L_{\xi,1}$ and $L_{\xi,2}$ in $[0,\infty)\times \mathbb{R}$ 
is called the \emph{combined graph} of the function $\uL_{\xi}$, and is denoted $\g_{\xi}$. 
Our objective is to show that the simple continued fraction expansion of a number 
$\xi \in \mathbb{R}$ can be read from this graph when $\xi \in [0,\frac{1}{2}$\normalsize$]$

\begin{maintheorem}
\label{imp1}
Let $\xi \in \mathbb{R}$. Then, $\g_{\xi}= \g_{\Vert \xi \Vert}$ depends only 
on the distance from $\xi$ to the closest integer, denoted $\Vert \xi \Vert$.
 Moreover, let $(q_{n})_{0\leqslant n < s}$ with
$s \in \mathbb{N}^{\star}\cup \lbrace \infty \rbrace$ denote the increasing 
sequence of points in $[0,\infty)$ for which $L_{\xi,1}$ admits a local maximum,
 and let $a_n$ denote for each positive integer $n<s$ the number of local maxima 
of $L_{\xi,2}$ restricted to the interval $[q_{n-1},q_{n}]$. Then it follows that $q_{0}=0$,
 and that the simple continued fraction expansion of $\Vert \xi \Vert$ is given by
 \begin{align}
\label{dv}
\Vert \xi \Vert=\begin{cases} [0] & \text{if} \quad s=1 ,\\
[0,a_{1}, a_{2},\ldots,a_{s-1}] & \text{if} \quad 2\leqslant s < \infty , \\
[0,a_{1},a_{2},\ldots] & \text{if} \quad s=\infty,
\end{cases}
\end{align}
with $a_{s-1}\geqslant 2$ if $2\leqslant s < \infty$.
\end{maintheorem}

In particular, the numbers $a_{n}$ are the \emph{partial quotients} of $\Vert \xi \Vert$. 
Figure 1 illustrates this result by showing the combined graph $\g_{\xi}$ (in solid lines) 
for several rational numbers $\xi$.

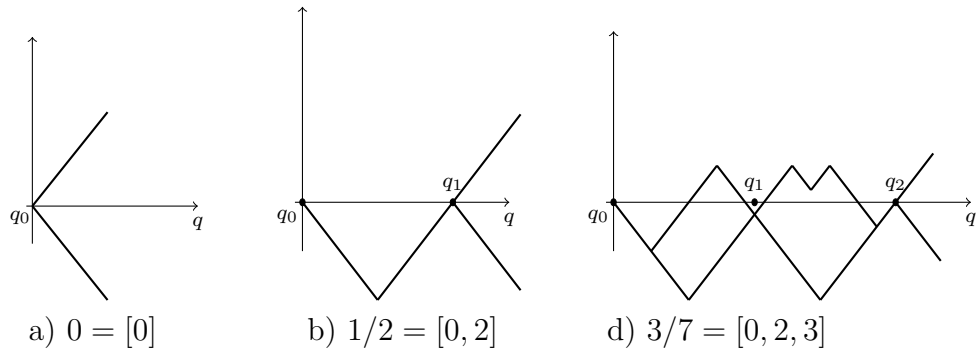
\begin{figure}[!h]
 \label{fig1}
 \begin{tabular}{lllllll}
\begin{tikzpicture}[xscale=0.4,yscale=0.5]
\draw[->] (-0.7,0)--(-0.5,0)  -- (5,0)node[below]{\tiny $q$};
\draw[->]  (-0.5,-1) --(-0.5,0)-- (-0.5,4.5);
\node at (-0.9,-0.3) {\tiny $q_{0}$};
\draw[thick,domain=-0.5:2] plot (\x, {-0.5-\x});
\draw[thick,domain=-0.5:2] plot (\x, {0.5+\x});

\end{tikzpicture}
&
\qquad
&
\begin{tikzpicture}[xscale=0.5,yscale=0.65]
\draw[->] (-0.7,0)--(-0.5,0)  -- (5,0)node[below]{\tiny $q$};
\draw[->]  (-0.5,-1) --(-0.5,0)-- (-0.5,4);
\node at (-0.9,-0.3) {\tiny $q_{0}$};

\draw[thick,domain=-0.5:1.5] plot (\x, {-0.5-\x});
\draw[  thick, domain=1.5:5.3] plot (\x, {-3.5+\x});
\draw[fill] (-0.5,0) circle [radius=2pt];

\draw[  thick,domain=3.5:5.3] plot (\x, {3.5-\x});

\draw[fill] (3.5,0) circle [radius=2pt]node[above]{\tiny $ q_{1}$ };

\end{tikzpicture}

&
\qquad
&

\begin{tikzpicture}[xscale=0.5,yscale=0.65]
\draw[->] (-0.7,0)--(-0.5,0)  -- (9,0)node[below]{\tiny $q$};
\draw[->]  (-0.5,-1) --(-0.5,0)-- (-0.5,3.5);
\node at (-0.9,-0.3) {\tiny $q_{0}$};

\draw[thick,domain=-0.5:1.5] plot (\x, {-0.5-\x});
\draw[  thick, domain=1.5:4.25] plot (\x, {-3.5+\x});
\draw[fill] (-0.5,0) circle [radius=2pt];

\draw[  thick,domain=3.25:5] plot (\x, {3-\x});
\draw[fill] (3.25,0) circle [radius=2pt]node[above]{\tiny $ q_{1}$ };

\draw[  thick, domain=0.5:2.25] plot (\x, {-1.5+\x});
\draw[thick,domain=2.25:3.25] plot (\x, {3-\x});

\draw[thick,domain=5:8] plot (\x, {-7+\x});
\draw[thick,domain=7:8.2] plot (\x, {7-\x});
\draw[fill] (7,0) circle [radius=2pt]node[above]{\tiny $ q_{2}$ };

\draw[thick, domain=4.25:4.75] plot (\x, {5-\x});
\draw[thick,domain=4.75:5.25] plot (\x, {-4.5+\x});
\draw[thick, domain=5.25:6.5] plot (\x, {6-\x});
\end{tikzpicture}

& 
\qquad \qquad

& 
\qquad \qquad
\end{tabular}

\begin{tabular}{llllllll}
 \,\, a) $0=[0]$\qquad \qquad \quad & b) $1/2=[0,2]$ \qquad \quad  &\,\,d) $3/7=[0,2,3]$ 
 & \qquad \qquad  &\qquad\qquad  &\qquad\qquad
  \end{tabular}
       
       \caption{ Examples of combined graphs.}
     
 \end{figure}
For the proof, fix an arbitrary $\xi \in \bR$, and choose $m \in \mathbb{Z}$ 
and $\epsilon=\pm 1$ such that $\xi = m + \epsilon \Vert \xi \Vert  $. The map
\begin{align*}
\phi : \quad \mathbb{R}^{2}&\longrightarrow \mathbb{R}^{2}\\
  \quad(x,y) &\longmapsto (x, \epsilon(y- m x))
\end{align*}
is an $\bR$-vector space isomorphism which satisfies $\phi(\mathbb{Z}^{2}) = \mathbb{Z}^{2}$.
 Thus, for each  $q\geqslant 0$, we have that
\[\lambda_{j}(\mathcal{C}_{\xi}(e^{q})) = \lambda_{j}(\phi(\mathcal{C}_{\xi}(e^{q}))
=\lambda_{j}(\mathcal{C}_{\Vert \xi \Vert}(e^{q}))  \qquad  (j=1, 2),\]
from which it follows that $\uL_{\xi}(q)= \uL_{\Vert \xi \Vert}(q)$. 
This proves the first statement of the theorem. 

For what follows, we shall suppose 
that $\xi=\Vert \xi \Vert \in [0,\frac{1}{2}$\normalsize$]$. The simple continued 
fraction expansion of $\xi \in [0,\frac{1}{2}$\normalsize$]$ is given by (\ref{dv}) 
for some sequence $(a_{n})_{1\leqslant n<s}$ in $\mathbb{N}^{\star}$, with  
$a_{s-1}\geqslant 2$ if $2\leqslant s < \infty$. The $n$-th  convergent of $\xi$ is the rational number
\[\cfrac{P_{n}}{Q_{n}}= [0, a_{1},\ldots, a_{n}] \qquad (1\leqslant n<s),\]
where $P_{n}$ and $Q_{n}$ are the positive integers given by the recurrence formula
\begin{align*}
Q_{n}&= Q_{n-2}+a_{n}Q_{n-1}\quad\text{and}\quad P_{n}= P_{n-2} +a_{n}P_{n-1} \qquad (1\leqslant n<s),
\end{align*}
with initial values $P_{-1}=Q_{0}=1$ and $P_{0}=Q_{-1}=0$. Let $Q_{s}=\infty$ if $s<\infty$.
 The proof of the second part of the theorem uses the following facts (see \cite[chap I]{Sc}).
\begin{enumerate}[label=(\roman*)]
 \item 
 The sequence $(Q_{n})_{0\leqslant n < s}$ is a strictly increasing sequence of positive integers.
 \item 
 The sequence $(Q_{n}\xi -P_{n})_{-1\leqslant n < s}$ consists of real numbers of alternating 
signs (except the last term which is zero if $s<\infty$), whose absolute values are strictly decreasing.
 \item 
We have $Q_{n}P_{n-1}-Q_{n-1}P_{n}=(-1)^{n}$  for each $n\in\bN$ satisfying $0\leqslant n < s$.
 \end{enumerate}
 
Consider the integers
\begin{align}
\label{ri}
Q_{n,t}&= Q_{n-2}+tQ_{n-1}\quad \text{and} \quad P_{n,t}= P_{n-2}+ tP_{n-1} 
\quad (0\leqslant t\leqslant a_{n} \, ,\, 1\leqslant n < s).
\end{align}
The fractions $P_{n,t}/Q_{n,t}$ with $0< t < a_{n}$, when they exist, are called the 
\emph{semiconvergents} of $\xi$ between $P_{n-1}/Q_{n-1}$ and $P_{n}/Q_{n}$. 
The following proposition relates the points $\ux_{n,t}=(Q_{n,t},P_{n,t})$ to the 
points $\ux_{n}=(Q_{n},P_{n})$ with  $-1\leqslant n<s$, as well as to the quantities 
$\Delta_{n}$ and $\Delta_{n,t}$ defined by
\[\Delta_{n}=|Q_{n}\xi -P_{n}| \quad (-1\leqslant n < s)\,\, \text{and} \,\,
\Delta_{n,t}=|Q_{n,t}\xi -P_{n,t}| \quad (0\leqslant t\leqslant a_{n} \, ,\, 1\leqslant n < s).\]

\begin{proposition}
\label{pr1}
Let $n$ be a positive integer with $n<s$. Then we have
\begin{align}
\label{eg1}
Q_{n,0}&=Q_{n-2} < Q_{n-1} \leqslant Q_{n,1}<\cdots\, < Q_{n,a_{n}}=Q_{n} ,\\
\label{eg2}
\Delta_{n,a_{n}}&=\Delta_{n}< \Delta_{n-1} \leqslant \Delta_{n,a_{n}-1}<\cdots < \Delta_{n,0}=\Delta_{n-2},
\end{align}
with $ Q_{n,1}= Q_{n-1}$ iff $n=1$, and $\Delta_{n-1} =\Delta_{n,a_{n}-1}$ iff
 $s<\infty$ and $n=s-1$. Moreover, the points $\ux_{n,t}=(Q_{n,t},P_{n,t})$ with
 $0\leqslant t\leqslant a_{n}$ are precisely the pairs of non-negative integers $(Q,P)$ satisfying
\begin{equation}
\label{eq1}
Q_{n-2}\leqslant Q\leqslant Q_{n},\quad |Q\xi -P| \leqslant \Delta_{n-2}=|Q_{n-2}\xi -P_{n-2}|,\quad QP_{n-1}-Q_{n-1}P\neq 0.
\end{equation}
Finally, if $s<\infty$, then there exists no pair $(Q,P) \in \bZ^{2}\backslash\{0\}$ satisfying
\begin{equation}
\label{eq2}
|Q\xi -P| < \Delta_{s-2}=|Q_{s-2}\xi -P_{s-2}|, \quad QP_{s-1}-Q_{s-1}P\neq 0.
\end{equation}

\end{proposition}

Due to the lack of a convenient reference, we give a succinct proof.

\begin{proof}
Since $Q_{n-2}\xi -P_{n-2}$ and $Q_{n-1}\xi -P_{n-1}$ are non-zero, of opposing signs, and since 
\[Q_{n}\xi -P_{n}=(Q_{n-2}\xi -P_{n-2})+a_{n}(Q_{n-1}\xi -P_{n-1})\] is either zero 
or of the same sign as $Q_{n-2}\xi -P_{n-2}$, it follows that 
\begin{align*}
\Delta_{n,t}=\Delta_{n-2}-t\Delta_{n-1} \quad (0\leqslant t\leqslant a_{n}),
\end{align*}
and that $\Delta_{n,a_{n}-1}=\Delta_{n,a_{n}}+\Delta_{n-1} = \Delta_{n}+\Delta_{n-1}$. 
The first assertion of the Proposition then follows from properties (i) and (ii).

Let $\ux=(Q,P)$ be a non-zero integer point satisfying (\ref{eq1}). Since  
$\ux_{n-2}=(Q_{n-2},P_{n-2})$ and $\ux_{n-1}=~(Q_{n-1},P_{n-1}) $ form a basis of 
$\mathbb{Z}^{2}$ (by virtue of (iii)), we may write
\[\ux=r\ux_{n-2}+t\ux_{n-1}.\]
with $(r,t) \in \mathbb{Z}^{2}\backslash \lbrace 0\rbrace $. The third condition of 
(\ref{eq1}) yields $r\neq 0$. Thus, the equality
\[ Q\xi -P=r(Q_{n-2}\xi -P_{n-2})+t(Q_{n-1}\xi -P_{n-1})\]
combined with property (ii) and the second condition of (\ref{eq1}) implies that
 $rt\geqslant 0$ and that $t>a_{n}$ if $r>1$. Since $Q= rQ_{n-2}+ tQ_{n-1}$, 
then (i) allows us to conclude that $r=1$ and that $0\leqslant t\leqslant a_{n}$.

Finally, if $s< \infty$ and $\ux=(Q,P)$ satisfies (\ref{eq2}), 
we may write $\ux=a\ux_{s-2}+b\ux_{s-1}$ with $a,b \in \bZ$ and $a\neq 0$. Since 
$\Delta_{s-1}=0$, we deduce that $|Q\xi -P|=a\Delta_{s-2}\geqslant \Delta_{s-2} $, 
contrary to the hypothesis.
\end{proof}

For each $\ux=(Q,P)\in\mathbb{Z}^{2}$ and each $q\geqslant 0$, we define $L_{\ux}(q)$ to be the logarithm 
of the smallest real number $\lambda \geqslant 0 $ such that $\ux \in \lambda  \mathcal{C}_{\xi}(e^{q})$. 
A simple computation gives
\[L_{\ux}(q)=\max \lbrace \log|Q|-q\, ,\, \log|Q\xi -P|+q\rbrace \qquad ( q\geqslant 0). \]
The graph of the function $L_{\ux}: [0,\infty) \rightarrow \mathbb{R}$ is called the
 \emph{trajectory} of the point $\ux$. The union of these graphs for $\ux\in \bZ^{2}\backslash \lbrace 0\rbrace$ 
primitive (i.e. with coprime coordinates) contains $\g_{\xi}$ (see \cite[$\mathsection 4$]{SS2009} or 
\cite[$\mathsection 2$]{R_preprint}).  
In fact the following smaller set of points suffices.

 \begin{proposition}
\label{pr2}
The graph $\g_{\xi}$ is covered by the trajectories of the points $\ux_{n}$ 
with $-1\leqslant n < s$ and those of the points $\ux_{n,t}$ with $1\leqslant n < s$ and $0< t< a_{n}$.
\end{proposition}

\begin{proof}
Indeed, let $\ux=(Q,P)\in\bZ^{2}$ be primitive  with $\pm\ux$ different from 
all these points. Then we have $Q\neq 0$ otherwise $\ux=(0,\pm 1)=\pm \ux_{-1}$ against the hypothesis. Since $L_{\ux}=L_{-\ux}$, 
we may assume without loss of generality that $Q\geqslant 1$. Thus, there exists an integer $n$ with $1\leqslant n \leqslant s$
such that
 $Q_{n-1} \leqslant Q < Q_{n}$, while Proposition \ref{pr1} yields $|Q\xi -P|\geqslant \Delta_{n-2}$. We deduce that 
 \[
L_{\ux}(q)\geqslant \max \left\lbrace  \log Q_{n-1} - q ,\,  \log \Delta_{n-2}+ q \right\rbrace 
=\max \lbrace L_{\ux_{n-2}}(q),L_{\ux_{n-1}}(q)\rbrace ,
\]
for all $q\geqslant 0$. Thus, since $\ux_{n-2}$ and $\ux_{n-1}$ are linearly independent, 
the trajectory of $\ux$ does not contribute to the cover.
\end{proof}

The inequalities of (\ref{eg1}) and (\ref{eg2}) of Proposition \ref{pr1} yield
\begin{align}
\label{ineg}
L_{\ux_{n,t}}(q)&\geqslant L_{\ux_{n-1}}(q) \qquad (q\geqslant 0,\,1\leqslant n < s,\,0< t< a_{n}). 
\end{align}
Since $L_{\ux_{-1}}(q)=q\geqslant L_{\ux_{0}}(q)$ for all $q\geqslant 0$, it follows 
that the graph of $L_{\xi,1}$ is covered by the trajectories of the points $\ux_{n}$ 
with $0\leqslant n< s$, and thus that
\[L_{\xi,1}(q)=  \min \left\lbrace  L_{\ux_{n}}(q) \, ; \, 0 \leqslant n <s \right\rbrace \qquad (q\geqslant 0).\]
In particular, we have $L_{\xi,1}(q)= L_{\ux_{0}}(q)=-q$ near $q=0$. Thus, $q_{0}=0$ is the
 first local maximum of $L_{\xi,1}$ on $[0,\infty)$.
Let $n$ be a positive integer with $ n<s$. The inequalities of (\ref{eg1}) and (\ref{eg2}) 
also imply that the trajectories of $\ux_{n-1}$ and of $\ux_{n}$ meet in a single point. 
To the left of this point, the trajectory of $\ux_{n-1}$ has slope $+1$, while to the right,
 the trajectory of $\ux_{n}$ has slope $-1$. The abscissa of this point is thus the $(n+1)$-th 
local maximum $q_{n}$ of $L_{\xi,1}$ on $[0,\infty)$. It is given by
\[
 q_{n}= \frac{1}{2}\left( \log Q_{n} - \log \Delta_{n-1}\right) .
\]
Moreover, $L_{\xi,1}$ coincides with $L_{\ux_{n-1}}$ on the interval $[q_{n-1},q_{n}]$ (see Figure 2). 
Since none of the points $\ux_{n,0}~=\ux_{n-2}$, $\ux_{n,1},\ldots, \ux_{n,a_{n}}=\ux_{n}$ is a 
multiple of $\ux_{n-1}$, we further deduce from (\ref{ineg}) that, on the same interval, we have
\begin{equation*}
\label{min}
L_{\xi,2}(q)=  \min \left\lbrace  L_{\ux_{n,t}}(q) \, ; \, 0 \leqslant t \leqslant a_{n} \right\rbrace 
\qquad (q_{n-1}\leqslant q\leqslant q_{n}).
\end{equation*}
If $n\neq 1$ and $n\neq s-1$, each inequality of (\ref{eg1}) and of (\ref{eg2}) is strict. Thus, 
the trajectories of the points $\ux_{n,t}$ with $0 \leqslant t \leqslant a_{n}$ are a part of the 
graph of $L_{\xi,2}$ over $[q_{n-1},q_{n}]$, as shown in Figure 2.b. The points at which the trajectories 
of $\ux_{n,t-1}$ and of $\ux_{n,t}$ cross for $t=1,\ldots,a_{n}$ give the $a_{n}$ local maxima of $L_{\xi,2}$
 on $[q_{n-1},q_{n}]$, each contained in the interior of this interval. If $n=1$ and $s\geqslant 2$,
 the situation is illustrated by Figure 2.a. In this case, $q_{0}=0$ is the first local maximum of
 $L_{\xi,2}$ on $[q_{0},q_{1}]$. If $s<\infty$ and if $n=s-1\geqslant 1$, then $q_{s-1}$ is the last
 local maximum of $L_{\xi,2}$ on $[q_{s-2},q_{s-1}]$, as illustrated in Figure 2.c. Finally, if $s=1$,
 we have that $\xi =0$ and the graph $\g_{0}$ is given by Figure 1.a. In general, if $s<\infty$, 
we have that $L_{\xi,2}(q)= L_{\ux_{s-2}}(q)= q+\log \Delta_{s-2}$ for all $q\geqslant q_{s-1}$, 
from which it follows that $L_{\xi,2}$ admits no local maximum on $[q_{s-1},\infty)$.

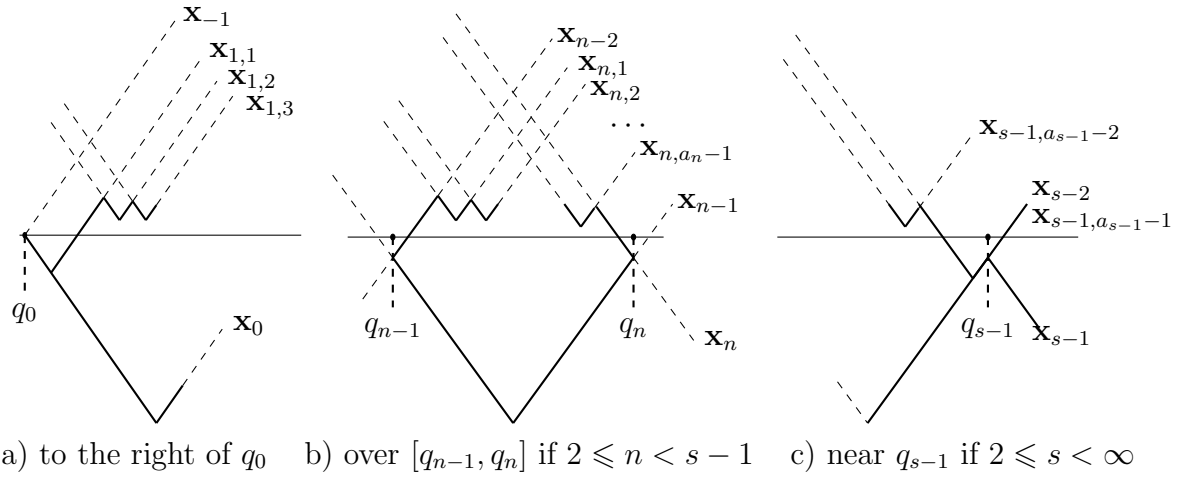
\begin{figure}[ht!]
 \label{fg2}
 
 \begin{tabular}{lll}

\begin{tikzpicture}[xscale=0.35,yscale=0.5]
\draw (-0.7,0)--(0,0)  -- (10,0);

\draw[thick,domain=-0.5:4.5] plot (\x, {-0.5-\x});
\draw[dashed, thin,domain=-0.5:5.3] plot (\x, {0.5+\x});

\draw[  thick, domain=4.5:5.5] plot (\x, {-9.5+\x});
\draw[dashed,  thin, domain=5.5:7] plot (\x, {-9.5+\x});
\draw[fill] (-0.5,0) circle [radius=2pt];

\draw [ thick, dashed] (-0.5,0)--(-0.5,-1.5)node[below]{$q_{0}$ };

\draw[  thick, domain=0.5:2.5] plot (\x, {-1.5+\x});
\draw[dashed,  thin, domain=2.5:6.3] plot (\x, {-1.5+\x});

\draw[ dashed, thin,domain=0.5:2.5] plot (\x, {3.5-\x});
\draw[thick,domain=2.5:3.1] plot (\x, {3.5-\x});
\draw[  thick, domain=3.1:3.6] plot (\x, {-2.7+\x});
\draw[dashed,  thin, domain=3.6:6.8] plot (\x, {-2.7+\x});

\draw[ dashed, thin,domain=1:3.6] plot (\x, {4.5-\x});
\draw[thick,domain=3.6:4.1] plot (\x, {4.5-\x});
\draw[thick, domain=4.1:4.5] plot (\x, {-3.7+\x});
\draw[dashed,  thin, domain=4.5:7.4] plot (\x, {-3.7+\x});

\draw[fill] (5.1,5.8)      node[right]{ $\ux_{-1} $};
\draw[fill]  (7,-2.4)      node[right]{ $\ux_{0} $};
\draw[fill]  (6.1,4.8)    node[right]{ $\ux_{1,1} $};
\draw[fill]  (7.5,3.4)     node[right]{$\ux_{1,3} $};
\draw[fill]  (6.8,4.1)     node[right]{$\ux_{1,2} $};
\end{tikzpicture}
&

 \begin{tikzpicture}[xscale=0.4,yscale=0.55]
\draw (-0.5,0)--(0,0)  -- (10,0);

\draw[dashed,  thin, domain=0:1] plot (\x, {-1.5+\x});
\draw[  thick, domain=1:2.5] plot (\x, {-1.5+\x});
\draw[dashed,  thin, domain=2.5:6.3] plot (\x, {-1.5+\x});

\draw[ dashed, thin,domain=-1:1] plot (\x, {0.5-\x});
\draw[thick,domain=1:5] plot (\x, {0.5-\x});

\draw[  thick, domain=5:9] plot (\x, {-9.5+\x});
\draw[dashed,  thin, domain=9:10.3] plot (\x, {-9.5+\x});
\draw[fill] (1,0) circle [radius=2pt];

\draw [ thick, dashed] (1,0)--(1,-1.7)node[below]{$ q_{n-1}$ };

\draw[dashed, thin,domain=3.1:7.77] plot (\x, {8.5-\x});
\draw[  thick,domain=7.77:9] plot (\x, {8.5-\x});
\draw[dashed, thin,domain=9:11] plot (\x, {8.5-\x});

\draw[fill] (9,0) circle [radius=2pt];
\draw [ thick, dashed] (9,0)--(9,-1.7)node[below]{$ q_{n}$ };

\draw[ dashed, thin,domain=0.7:2.5] plot (\x, {3.5-\x});
\draw[thick,domain=2.5:3.1] plot (\x, {3.5-\x});
\draw[  thick, domain=3.1:3.6] plot (\x, {-2.7+\x});
\draw[dashed,  thin, domain=3.6:6.8] plot (\x, {-2.7+\x});

\draw[ dashed, thin,domain=1.2:3.6] plot (\x, {4.5-\x});
\draw[thick,domain=3.6:4.1] plot (\x, {4.5-\x});
\draw[thick, domain=4.1:4.5] plot (\x, {-3.7+\x});
\draw[dashed,  thin, domain=4.5:7.4] plot (\x, {-3.7+\x});

\draw[ dashed, thin,domain=2.67:6.7] plot (\x, {7.5-\x});
\draw[thick,domain=6.7:7.25] plot (\x, {7.5-\x});

\draw[  thick, domain=7.25:7.7] plot (\x, {-7+\x});
\draw[dashed,  thin, domain=7.7:9.05] plot (\x, {-7+\x});

\draw[fill] (6.1,4.8)   node[right]{ $\ux_{n-2} $};
\draw[fill] (11,-2.5)   node[right]{ $\ux_{n} $};
\draw[fill] (10.1,0.8)    node[right]{ $\ux_{n-1} $};
\draw[fill] (8.85,2.05)   node[right]{ $\ux_{n,a_{n}-1} $};
\node at (8.9,2.7)     { $\cdots $};
\draw[fill] (7.2,3.5)     node[right]{$\ux_{n,2} $};
\draw[fill] (6.8,4.1)   node[right]{$\ux_{n,1} $};

\end{tikzpicture}
&

\begin{tikzpicture}[xscale=0.4,yscale=0.55]
\draw (2,0) -- (12,0);

\draw[ dashed, thin,domain=4:5] plot (\x, {0.5-\x});
\draw[  thick, domain=5:9] plot (\x, {-9.5+\x});
\draw[ thick, domain=9:10.3] plot (\x, {-9.5+\x});

\draw[thick,domain=9:10.7] plot (\x, {8.5-\x});
\draw[fill] (9,0) circle [radius=2pt];
\draw [ thick, dashed] (9,0)--(9,-1.7)node[below]{$ q_{s-1}$ };

\draw[ dashed, thin,domain=2.61:6.7] plot (\x, {7.5-\x});
\draw[thick,domain=6.7:8.5] plot (\x, {7.5-\x});
\draw[  thick, domain=8.5:9] plot (\x, {-9.5+\x});
\draw[dashed,  thin, domain=9:10.3] plot (\x, {-9.5+\x});

\draw[ dashed, thin,domain=2.2:5.7] plot (\x, {6.5-\x});
\draw[thick,domain=5.7:6.25] plot (\x, {6.5-\x});
\draw[  thick, domain=6.25:6.7] plot (\x, {-6+\x});
\draw[dashed,  thin, domain=6.7:8.55] plot (\x, {-6+\x});

\draw[fill] (12.8,-2.4) node[left]{ $\ux_{s-1} $};
\draw[fill] (10.1,0.4)    node[right]{$\ux_{s-1,a_{s-1}-1} $};
\draw[fill] (8.35,2.55)    node[right]{ $\ux_{s-1,a_{s-1}-2} $};
\draw[fill] (10.1,1.1)  node[right]{ $\ux_{s-2} $};

\end{tikzpicture}

\end{tabular}

\begin{tabular}{lll}
 a) to the right of $q_{0}$ & \ b) over $[q_{n-1},q_{n}]$ if $2\leqslant n< s-1$ &  \ c) near $ q_{s-1}$ if $2\leqslant s< \infty$
  \end{tabular}

       \caption{Combined graph $\g_{\xi}$. }

 \end{figure}

\end{document}